\documentclass{amsart}
\usepackage{graphicx} % Required for inserting images
\usepackage{amssymb}
\usepackage{amsmath}
\usepackage{amsthm}
\usepackage{mathtools}
\usepackage{graphicx}
\usepackage{tikz}
\usepackage{comment}
\usepackage{hyperref}

\hypersetup{
    colorlinks=true,
    linkcolor=red,
}

\newtheorem{theorem}{Theorem}[section]
\newtheorem{proposition}[theorem]{Proposition}

\newtheorem{conjecture}[theorem]{Conjecture}

\newtheorem{claim}[theorem]{Claim}
\newtheorem{remark}[theorem]{Remark}
\newtheorem{definition}[theorem]{Definition}

\newcommand{\R}{\mathbb R}
\newcommand{\F}{\mathbb F}

\title{Bourgain-type projection theorems over finite fields}
\author{Alex Rose}
\date{}

\begin{document}

\begin{abstract}
We prove finite-field analogs of Bourgain's projection theorem in higher dimensions. In particular, for a certain range of parameters we improve on an exceptional set estimate by Chen in all dimensions and codimensions.
\end{abstract}

\maketitle

\section{Introduction}
\subsection{Over real numbers}
Fractal dimensions of orthogonal projections over real numbers were extensively studied. For example, the following bounds on the dimension of the exceptional set were obtained in works by Kaufman \cite{Kau68}, Mattila \cite{Mat75}, Falconer \cite{Fal82}, Peres and Schlag \cite{Per00}.

\begin{theorem}\label{Kaufman}
    Let $K \subseteq \mathbb{R}^n$ be a Borel set. The following bounds hold: 
    \begin{itemize}
        \item If $\dim K \leq m$ and $0 < t \leq \dim K$, then 
        $$ \dim\{W \in Gr(n, m) \mid \dim \pi_W(K) \leq t\} \leq m(n - m) - (m - t).$$
        \item If $\dim K > m$, then 
        $$ \dim\{W \in Gr(n, m) \mid \mathcal{H}^m(\pi_W(K)) = 0\} \leq m(n - m) - (\dim K - m).$$
        \item If $\dim K > 2m$, then
        $$ \dim\{W \in Gr(n, m) \mid Int(\pi_W(K)) = \varnothing\} \leq m(n - m) - (\dim K - 2m).$$
    \end{itemize}
\end{theorem}

Here $Gr(n, m)$ denotes the set of $m$-dimensional subspaces of the $n$-dimensional space, $\pi_W$ denotes the orthogonal projection onto $W$ and $Int$ denotes the interior of a set.

In particular, in the plane the first bound above yields that $$\dim\{ \theta \in S^1 \mid \dim \pi_\theta(K) < t\} \leq t.$$ It was conjectured by Oberlin \cite{Obe12} that for $\dim K \leq 1$ and $t \in (\frac{1}{2} \dim K, \dim K)$ the sharp bound for the planar case is actually $2t - \dim K$. Bourgain, in his groundbreaking paper \cite{Bou10} established the $t = \dim K/2$ case of this conjecture.

\begin{theorem}\label{BourgainProjectionTheorem}
    For any $\kappa \in (0, 1)$ and $\alpha \in (0, 2)$ there exists $\varepsilon = \varepsilon(\kappa, \alpha) > 0$ such that for any $K$ with $\dim K = \alpha$ the following holds:
    $$\dim \{\theta \in S^1 \mid \dim \pi_\theta(K) < \frac{\alpha}{2} + \varepsilon\} < \kappa.$$ 
\end{theorem}

Recently, the conjecture was proven, first for almost Ahlfors-David regular sets by Orponen and Shmerkin \cite{Orp23}, and then for all Borel sets by Ren and Wang \cite{Ren23}.

In higher dimensions, He \cite{He18} obtained analogs of Bourgain's projection theorem for all codimensions.

\begin{theorem}\label{higherBourgainProjection}
    For any $\kappa \in (0, 1)$ and $\alpha \in (0, n)$ there exists $\varepsilon = \varepsilon(\kappa, \alpha) > 0$ such that for any analytic set $K \subseteq \R^n$ with $\dim K = \alpha$ the following holds:
    the set $\{W \in Gr(n, m) \mid \dim \pi_W(K) < \frac{m}{n}\alpha + \varepsilon\}$ does not support any non-zero measure $\mu$ on $Gr(n, m)$ with the following non-concentration property:
    $$\forall \rho > 0, \forall V \in Gr(n, n - m):  \mu(\{W \mid \angle(W, V) \leq \rho \}) \leq \rho^\kappa.$$
\end{theorem}

\subsection{Over finite fields}
In this paper we are concerned with finite-field analogs of the abovementioned results. We follow the notation of Chen \cite{Che17} and define projections in terms of the kernel.

\begin{definition}
    Let $\F$ be a field. For a linear subspace $W \subseteq \F^n$ we define the projection map by $\pi^W(x) := x + W \in \F^n/W$. 
\end{definition}
In the case of real numbers, $\pi^{W}$ coincides with $\pi_{W^\bot}$ considered above. Chen in \cite{Che17} initiated the study of exceptional sets for projections over finite fields. He obtained the following bounds:
\begin{theorem}[Corollary 1.3 in \cite{Che17}] \label{ChenTheorem}
Let $K \subseteq \F_p^n$ with $|K| = p^s$. Then 
\begin{itemize}
    \item If $s \leq m$ and $t \in (0, s]$, then
    $$|\{W \in Gr(n, n - m) \mid |\pi^{W}(K)| \leq p^t/10\}| \leq \frac{1}{2}p^{m(n - m) - (m - t)}$$
    \item If $s > m$, then
    $$|\{W \in Gr(n, n - m) \mid |\pi^{W}(K)| \leq p^m/10\}| \leq \frac{1}{2}p^{m(n - m) - (s - m)}$$
    \item If $s > 2m$, then
    $$|\{W \in Gr(n, n - m) \mid |\pi^{W}(K)| \neq p^m\}| \leq 4p^{m(n - m) - (s - 2m)}$$
\end{itemize}
\end{theorem}
Here and throughout the paper $\F_p$ denotes the field with $p$ elements and $p$ is assumed to be a prime number. Chen also conjectured an analog of the sharp projection theorem in $\mathbb{R}^2$. We slightly reformulate his conjecture below:

\begin{conjecture}\label{planarConjecture}
    Let $K \subseteq \F_p^2$ and $E \subseteq Gr(2, 1)$. Assume that  $0 \leq |K| \leq p$ and $2 \leq |E| \leq |K|$. Then there exists $W \in E$ with $$|\pi^W(K)| \gtrsim |K|^\frac{1}{2}|E|^\frac{1}{2}.$$
\end{conjecture}

Bright and Gan in \cite{Bri23} showed that this lower bound can be attained. They also improved the bound on the size of the exceptional set when $|K|$ is large (for arbitrary $n$). Fraser and Rakhmonov in their paper \cite{Fra25} studied the projections of sets $K$ having certain Fourier analytic properties and in some cases obtained improved bounds on the size of the exceptional set. Lund, Pham and Vinh in a recent paper \cite{Lun23} proved the following bound on the size of the projections in the plane, improving the results of Chen.

\begin{theorem}\label{planarProjections}
    Let $K \subseteq \F_p^2$ and $E \subseteq Gr(2, 1)$. Assume that $|K|^\frac{1}{2} \leq |E| \leq |K| \leq p$. Then there exists $W \in E$ such that $$|\pi^W(K)| \gtrsim \max\{|K|^{\frac{1}{2}}|E|^{\frac{1}{6}}, |K|^\frac{2}{5}|E|^\frac{2}{5}, |E|\}.$$ 
\end{theorem}

\subsection{Main results}
The first result of this paper proves an analog of the bound in Theorem \ref{planarProjections} for one-dimensional projections in higher dimensions. We obtain the following Bourgain-type estimates:
\begin{theorem}\label{lineProjections}
    Consider $n \geq 2$. Let $K \subseteq \F_p^n$ and let $E \subseteq Gr(n, n - 1)$ be a set of hyperplanes that do not contain a common line (i.e. $\bigcap_{W \in E} W = 0$). Assume that $|K| |E|^\frac{2n + 1}{4(n - 1)} \leq p^n$. Then there exists $W \in E$ such that $$|\pi^W(K)| \gtrsim \min \{|K|^\frac{1}{n}|E|^\frac{1}{4n(n - 1)}, |K|^\frac{1}{n - 1}\}.$$
\end{theorem}

Note that if the set $E$ has cardinality larger than $O(p^{n - 2})$, then the condition that the elements of $E$ do not contain a common line is automatically satisfied. In this case the lower bound $|K|^\frac{1}{n}|E|^\frac{1}{4n(n - 1)}$ is an improvement on the lower bound $\gtrsim p^{2 - n}|E|$ provided by the first estimate in Chen's theorem when $|E| \ll |K|^\frac{4(n - 1)}{4n(n - 1) - 1} p^\frac{4n(n - 1)(n - 2)}{4n(n - 1) - 1}$.

To extend the theorem to other codimensions we will need the following definition.

\begin{definition}
    A set $E \subseteq Gr(n, m)$ is called non-degenerate if for every $V \in Gr(n, n - m)$ there exists $W \in E$ such that $W \cap V = 0$.
\end{definition}

We give several examples to illustrate the non-degeneracy condition.
\begin{itemize}
    \item Let $v_1, \ldots, v_n$ be a basis and let $E \subseteq Gr(n, m)$ be a set containing the subspace $\text{span} \{v_{i_1}, \ldots, v_{i_m}\}$ for any m-subset $\{i_1, \ldots, i_m\} \subseteq \{1, \ldots, n\}$. Then $E$ is non-degenerate. This is because given any $V \in Gr(n, n - m)$ one can construct the desired $W \in E$ by starting with $W = 0$ and then adding the elements of the basis to $W$ one by one while keeping the intersection $W \cap V$ trivial.
    
    \item Take $k$ sufficiently large depending only on $n$ and let $W_1, \ldots, W_k$ be independent uniformly distributed elements of $Gr(n, m)$. Then with probability $1 - o(1)$ the set $\{W_1, \ldots, W_k\}$ (and any set containing it) is non-degenerate. Indeed, a calculation shows that for any $V \in Gr(n, n - m)$ the probability that it non-trivially intersects $W_1$ is $\mathbb{P}(W_1 \cap V \neq 0) = O(p^{-1})$ (see Proposition \ref{subspaceCounting}). Hence, the probability that $W_i \cap V \neq 0$ for all $i$ is $O(p^{-k})$, and if $k$ is chosen larger than $m(n - m)$, then the probability that such $V$ exists is at most $O(p^{-k}|Gr(n, n - m)|) = o(1)$.
    
    \item Assume the set $E$ satisfies the analog of the non-concentration assumption in Theorem \ref{higherBourgainProjection}, that is, for any $V \in Gr(n, n - m)$ we have $|\{W \in E \mid W \cap V \neq 0\}| \leq p^{-\kappa}|E|$. Then the set $E$ is non-degenerate (trivially).
\end{itemize}
Now we can state the analog of Theorem \ref{lineProjections} for projections on higher-dimensional subspaces.
\begin{theorem}\label{BourgainTypeProjections}
Let $n \geq 2$ and let $0<\varepsilon < \varepsilon_0 = \frac{1}{4n(n - 1)(2n)^{n - 2}}$. Let $K \subseteq \F_p^n$ and  $E \subseteq Gr(n, n - m)$ be a non-degenerate set. Assume that $|K| |E|^{\frac{2n + 1}{4(n - 1)n^{n - 2}}} \leq p^n$ and $|E|^{\frac{1}{4 n^{n - 2}}} \leq |K|$. Then there exists a divisor $d \neq 1$ of $n$ that depends only on $K$ and if $d$ does not divide $m$, then 
$$\exists W \in E \text{ such that } |\pi^{W}(K)| \gtrsim |K|^{\frac{m}{n}}|E|^{\varepsilon}.$$
In particular, the estimate above holds if $n$ and $m$ are coprime.
\end{theorem}

We believe that the methods of \cite{He18} can be extended over finite fields. In light of this the main advantages of Theorem \ref{BourgainTypeProjections} over (the finite field version of) Theorem \ref{higherBourgainProjection} are:
\begin{itemize}
    \item The non-degeneracy assumption is much weaker than the non-concentration assumption. In fact, it suffices to produce finitely many elements of the set $E$ that together form a non-degenerate set, while the remaining elements of $E$ can be completely arbitrary.
    \item The dependence of the size of the projection on the cardinality of $E$ is explicit as a power of $|E|$.
    \item Explicit value of $\varepsilon_0 = \varepsilon_0(n)$.
\end{itemize}

In Chen's theorem (Theorem \ref{ChenTheorem}) the factor of $p^{m(n - m) - m}$ in the first estimate comes from the fact that the number of $(n - m)$-dimensional subspaces containing a given line is $(1 + o(1))p^{m(n - m) - m}$. While the elements of $E$ not containing a common line is not enough to obtain the full strength Bourgain-type bounds as in Theorem \ref{BourgainTypeProjections}, we can still improve over the first estimate in Theorem \ref{ChenTheorem}.

\begin{theorem} \label{improvementTheorem}
    Fix $\delta > 0$ and an integer $d \geq 1$. Let $E \subseteq Gr(n, n - m)$ be a set of subspaces and let $K \subseteq \F_p^n$. Assume that $|E| \geq 2p^{m(n - m) - m + d - 1}$ and $p^\delta \leq |K| \leq p^{d + 1- \delta}$. Then there exists $\varepsilon = \varepsilon(n, \delta) > 0$ and $W \in E$ such that $$|\pi^{W}(K)| \gtrsim |K|^\frac{d}{1 + d(n - m)}(|E|/p^{m(n - m) - m})^\varepsilon.$$
\end{theorem}

As we mentioned above, if we take $d = 1$ this theorem improves the bound $\frac{|E|}{4p^{m(n - m) - m}}$ from Theorem \ref{ChenTheorem} in the range $|E|/p^{m(n - m) - m} \ll |K|^\frac{1}{(1 - \varepsilon)(1 + n - m)}$. The condition $|K| \leq p^{d + 1 - \delta}$ is needed to rule out the following example: $K = (d + 1)$-subspace, $E = $ lines inside $K$, in which case $|\pi^{W}(K)| = |K|^\frac{d}{d + 1}$ for any $W \in E$.

The last result of the paper establishes 
Conjecture \ref{planarConjecture} for very small sets using the results of Grosu \cite{Gro13}.

\begin{theorem}\label{smallSetProjections}
    Let $K \subseteq \F_p^2$ and $E \subseteq Gr(2, 1)$ satisfy the bound $2 \leq |E|, |K| \leq \frac{1}{5} (\log_2 \log_6 \log_{18} p - 1)$. Then there exists $W \in E$ such that $$|\pi^W(K)| \gtrsim \min\{|K|^\frac{1}{2}|E|^\frac{1}{2}, |K| \}.$$
\end{theorem}
This theorem relies on the the finite field version of the celebrated Szemeredi-Trotter bound \cite{Sze83} obtained in \cite{Gro13} for very small sets. In fact, it is conjectured that the assumptions on the size of the set in this incidence bound can be significantly relaxed which, if proven, would imply the sharp projection bound for sets of larger cardinality. We refer the reader to \cite{Bou04}, \cite{Jon11}, \cite{Ste16}, \cite{Vin11} for partial results on this conjecture.

\section{Preliminaries and the outline of the proof}
Notation: throughout the paper, we write 
\begin{itemize}
    \item $A \lesssim B$ if $A \leq CB$,
    \item $A \lessapprox B$ if $A \lesssim (\log p)^{C} B$
\end{itemize}
where $C$ is some constant that may depend only on the dimension $n$ but not on $p$, $K$ or $E$.

We will make use of an analog of the inner product on $\F_p^n$.
\begin{definition}
For $x = (x_1, \ldots, x_n) \in \F_p^n$ and $y = (y_1, \ldots, y_n) \in \F_p^n$ define $\langle x, y \rangle = \sum_{i = 1}^n x_i y_i$.
\end{definition}
\begin{definition}
Let $S \in \F_p^n$ be a subset. We define the orthogonal complement as $S^\bot := \{x \in \F_p^n \mid \langle x, y \rangle = 0 \text{ for all } y \in S\}$.
\end{definition}
It follows from the fact that $\langle \bullet, \bullet \rangle$ is a non-degenerate bilinear form that for linear subspaces $W, W_1, W_2 \subseteq \F_p^n$ we have:
\begin{itemize}
    \item $\dim W^\bot = n - \dim W$
    \item $(W^\bot)^\bot = W$
    \item $(W_1 + W_2)^\bot = W_1^\bot \cap W_2^\bot$
    \item $(W_1 \cap W_2)^\bot = W_1^\bot + W_2^\bot$.
\end{itemize} 
The key tool in the proof of the $m = 1$ case of the main theorem is the following incidence bound of Stevens \cite{Ste16}:

\begin{theorem}\label{IncidenceBoundsTheorem}
    Let $\mathcal{L}$ be a set of lines in $\F_p^2$ and let $A, B \subseteq \mathbb{F}_p$ be sets satisfying $|A| \leq |B|$, $|A||B|^2 \leq |\mathcal{L}|^3$ and $|A||\mathcal{L}| \lesssim p^2$. Then $$I(A \times B, \mathcal{L}) \lesssim |A|^\frac{3}{4}|B|^\frac{1}{2}|\mathcal{L}|^\frac{3}{4} + |\mathcal{L}|$$
\end{theorem}

We will also need the following elementary propositions.

\begin{proposition}\label{niceBasis}
    If hyperplanes $W_1, \ldots, W_n \in Gr(n, n - 1)$ satisfy $W_1 \cap \ldots \cap W_n = 0$, then there exists a basis $v_1, \ldots, v_n$ of $\F_p^n$ satisfying 
    $$v_i \in W_j \Leftrightarrow i \neq j.$$
\end{proposition}

\begin{proof}
    Observe that for any $i$ we have $\bigcap_{j \neq i} W_j \not \subseteq W_i$ (otherwise the condition $W_1 \cap \ldots \cap W_n = 0$ is not satisfied). Fix some $v_i \in \bigcap_{j \neq i} W_j \setminus W_i$. Then the desired property $v_i \in W_j \Leftrightarrow i \neq j$ is clearly satisfied. To see that $v_1, \ldots, v_n$ is a basis note that for any $1 \leq i \leq n$ we have $\text{span}_{j \neq i} v_j \subseteq W_i$ while $v_i \not \in W_i$.
\end{proof}

The second elementary proposition bounds the size of the projection with fiber $W_1 \cap W_2$ in terms of the sizes of the projections with fibers $W_1$ and $W_2$.

\begin{proposition}\label{intersectionBound}
    Let $W_1, W_2$ be subspaces in $\F_p^n$ and let $K \subseteq \F_p^n$. Then $$|\pi^{W_1 \cap W_2}(K)| \leq |\pi^{W_1}(K)| |\pi^{W_2}(K)|.$$
\end{proposition}

\begin{proof}
    Let $M_1 := |\pi^{W_1}(K)|$ and $M_2 := |\pi^{W_1}(K)|$. Then for some $x_1, \ldots, x_{M_1} \in \F_p^n$ we have $K \subseteq \bigcup_{i = 1}^{M_1} (x_i + W_1)$. At the same time, for some $y_1, \ldots, y_{M_2} \in \F_p^n$ we have $K \subseteq \bigcup_{j = 1}^{M_2} (y_j + W_2)$. Thus $$K \subseteq (\bigcup_{i = 1}^{M_1} (x_i + W_1)) \cap (\bigcup_{j = 1}^{M_2} (y_j + W_2)) = \bigcup_{i, j} ((x_i + W_1) \cap (y_j + W_2)).$$
    Each of the sets $(x_i + W_1) \cap (y_j + W_2)$ is either empty or a translate of $W_1 \cap W_2$. Therefore, we see that $K$ can be covered by at most $M_1M_2$ translates of $W_1 \cap W_2$ and, hence, $|\pi^{W_1 \cap W_2}(K)| \leq M_1M_2$.
\end{proof}

The proofs of the main results proceed as follows: first we establish the $m = 1$ case (Theorem \ref{lineProjections}). After that we apply reductions similar to the one in \cite{He18} to establish Theorem $\ref{BourgainTypeProjections}$. Finally, from that we infer Theorem $\ref{improvementTheorem}$.

To prove the abovementioned reductions it will be convenient to introduce the following notation. For dimensions $n \geq 2$ and $1 \leq m \leq n - 1$, parameters $t, \varepsilon > 0$ and a subset $K \subseteq \F_p^n$ let $\mathcal{P}(m, t, \varepsilon)$ denote the following statement: \\
\textbf{For any non-degenerate set $E \subseteq Gr(n, n - m)$ with $|E| \geq p^t$ there exists a subset $K' \subseteq K$ and $W \in E$ such that $$|\pi^{W}(K'')| \gtrapprox|K|^\frac{m}{n} p^{\varepsilon t} \text{ for any } K'' \subseteq K' \text {with } |K''| \gtrapprox |K'|.$$}

To ease the notation we will write $\mathcal{P}(m, t, \varepsilon)$ and suppress the dependence on $K$ and $n$ since they will be assumed to be fixed. The proof of Theorem \ref{BourgainTypeProjections} proceeds as follows: first it follows from (the proof of) Theorem \ref{lineProjections} that for a small enough $\varepsilon = \varepsilon(n)$ the statement $\mathcal{P}(1, t, \varepsilon)$ holds. 

\begin{proposition}\label{baseCase}
    Let $n \geq 2$ and $\varepsilon = \frac{1}{4n(n - 1)}$. Then for any $K \geq \F_p^n$ and $t \geq 0$ satisfying $p^{t/4} \leq |K| \leq p^{n - \frac{2n + 1}{4(n - 1)}t}$ the statement $\mathcal{P}(1, t, \varepsilon)$ holds. 
\end{proposition}

Next we prove two reductions which slightly generalize the ones in \cite{He18}.
The first one is an analog of Proposition 10 from \cite{He18}. It allows us to sum bad dimensions.

\begin{proposition}\label{additionofDimensionsProposition}
    Fix $K \subseteq \F_p$ and $1 \leq m_1, m_2 \leq n - 1$ satisfying $m_1 + m_2 < n$. Then if $\mathcal{P}(m_1, t, \varepsilon)$ and $\mathcal{P}(m_2, t, \varepsilon)$ fail, then $\mathcal{P}(m_1 + m_2, t/n, 2n\varepsilon)$ also fails.
\end{proposition}

The second reduction allows us to sum bad codimensions. It is an analog of Proposition 11 from \cite{He18}.
\begin{proposition}\label{additionofCodimensionsProposition}
Fix $K \subseteq \F_p$ and $1 \leq m_1, m_2 \leq n - 1$ satisfying $m_1 + m_2 > n$. Then if $\mathcal{P}(m_1, t, \varepsilon)$ and $\mathcal{P}(m_2, t, \varepsilon)$ fail, then $\mathcal{P}(m_1 + m_2 - n, t/n, 2n\varepsilon)$ also fails.   
\end{proposition}

\begin{remark}
    The definition of $\mathcal{P}(m, t, \varepsilon)$ uses the notation $\lessapprox$. It will be clear from the proof of Proposition \ref{additionofDimensionsProposition} that the implicit powers of $\log p$ for which $\mathcal{P}(m_1 + m_2, t/n, 2n\varepsilon)$ fails will depend in a controlled way on the implicit powers of $\log p$ for which $\mathcal{P}(m_1, t, \varepsilon)$ and $\mathcal{P}(m_2, t, \varepsilon)$ fail. The same is true of Proposition \ref{additionofCodimensionsProposition}. In the end it will allows us to apply these propositions $O(1)$ times while keeping the implicit powers of the logarithm of the form $(\log p)^{O(1)}$.
\end{remark}

\section{Proof of the \texorpdfstring{$m = 1$}{m = 1} case}
We begin by proving Theorem \ref{lineProjections}. The proof considers two-dimensional slices of the set $K$ and applies the incidence bound to each of them.

\begin{proof}[Proof of Theorem \ref{lineProjections}]
    Let $n \geq 2$ and $E \subseteq Gr(n, n - 1)$. Assume that $E$ satisfies the condition $\bigcap_{W \in E}W = 0$. Fix the subspaces $W_1, \ldots, W_n \in E$ with $W_1 \cap \ldots \cap W_n = 0$. Proposition \ref{niceBasis} implies that there exists a basis $v_1, \ldots, v_n$ satisfying $v_i  \in W_j$ if and only if $i \neq j$. Applying a linear transformation that sends $v_1, \ldots, v_n$ to the standard basis we can assume that $v_i = (0, \ldots, 0, 1 ,0, \ldots, 0)$ (where the $1$ is in the $i$-th position). \\
    Next, assume for the sake of contradiction that $|\pi^W(K)| \leq M$ for some integer $M \leq |K|^\frac{1}{n}|E|^\frac{1}{4n(n - 1)}$ and all $W \in E$. In particular, taking $W = W_i$ we see that $K$ can be covered by at most $M$ copies of $\F_p^{i - 1} \times \{0\} \times \F_p^{n - i}$. Therefore, we obtain that $K$ is contained in a cartesian product of controlled size: $K \subseteq A_1 \times \ldots \times A_n$ for some sets $A_1, \ldots, A_n \subseteq \F_p$ satisfying $|A_i| = M$. 
    
    The next step is to find a coordinate plane spanned by $\{v_i, v_j\}$ for some $i \neq j$ and apply the incidence bound in this plane. To this end, observe that each $W \in E$ contains at most $n - 1$ of the vectors $v_1, \ldots, v_n$. Therefore, there exists a vector $v_i$ such that $W \not \ni v_i$ for at least $\gtrsim |E|$ subspaces $W \in E$. Without loss of generality we may assume that this holds for all $W \in E$ and that $i = 1$. Thus, for any $j \geq 2$ and $W \in E$ the intersection $W \cap  span \{v_1, v_j\}$ is a line not containing $v_1$. Moreover, it is easy to see that these intersections uniquely determine $W$ (to be more precise, $W = \oplus_{j \geq 2} (W \cap  span \{v_1, v_j\})$. Therefore, we can conclude that there exists $j \geq 2$ such that $$|\{W \cap span \{v_1, v_j\} \mid W \in E \}| \gtrsim |E|^\frac{1}{n - 1}.$$ Again, without loss of generality we can assume that $j = 2$. Then let $E'$ be a subset of $E$ of cardinality $|E|^\frac{1}{n - 1}$ such that the lines $\{W \cap \text{span}\{v_1, v_2\} \mid W \in E'\}$ are all distinct. 

    For each $W \in E'$ the bound $|\pi^{W}(K)| \leq M$ implies that $K$ can be covered by at most $M$ translates of $W$. Consequently, for each $x \in \F_p^n$ the slice $K_x := K \cap (x + \text{span} \{v_1, v_2\})$ can be covered by $M$ translates of the line $W \cap \text{span}\{v_1, v_2\}$. Let $\mathcal{L}_{x, W}$ be the set of these translates that cover $K_x$ so that $|\mathcal{L}_{x, W}| = M$ and define $\mathcal{L}_x := \bigcup_{W \in E'} \mathcal{L}_{x, W}$ to be the collection of all these lines. Observe that $|A_1| = |A_2| = M, |\mathcal{L}_x|^3 = |E'|^3 M^3 \geq |A_1||A_2|^2$ and $|A_1||\mathcal{L}_x| = M^2|E'| \leq p^2$ (the last inequality holds because of the assumption $|K| |E|^\frac{2n + 1}{4(n - 1)} \leq p^n$). Therefore, Theorem \ref{IncidenceBoundsTheorem} gives an upper bound
    $$\mathcal{I}(A_1 \times A_2, \mathcal{L}_x) \lesssim |A_1|^\frac{3}{4}|A_2|^\frac{1}{2}|\mathcal{L}|^\frac{3}{4} + |\mathcal{L}| \leq  M^2 |E'|^\frac{3}{4} + M|E'|.$$
    On the other hand, we have the lower bound $$\mathcal{I}(A_1 \times A_2, \mathcal{L}_x) \geq \sum_{W \in E'} \mathcal{I}(A_1 \times A_2, \mathcal{L}_{x, W}) \geq \sum_{W \in E'} |K_x| = |E'||K_x|.$$
    Recall that $K \subseteq A_1 \times \ldots \times A_n$ with $|A_i| \leq M$, so at most $M^{n - 2}$ of the slices $K_x$ are non-empty. Summing the upper and lower bounds over all $x \in \text{span} \{v_3, \ldots, v_n\}$ with $K_x \neq \varnothing$ we obtain the inequality 
    $$|E'||K| \lesssim M^n |E'|^\frac{3}{4} + M^{n - 1}|E'|.$$
    Rearranging, we get 
    $$M \gtrsim \min\{|K|^\frac{1}{n}|E'|^\frac{1}{4n}, |K|^\frac{1}{n - 1}\} \gtrsim \min\{|K|^\frac{1}{n}|E|^\frac{1}{4n(n - 1)}, |K|^\frac{1}{n - 1}\}.$$
\end{proof}

Finally, we prove that $\mathcal{P}(1, \varepsilon, t)$ holds. The only difference from the previous proof comes from the fact that now we are taking into account projections of subsets rather than the whole set.

\begin{proof}[Proof of Proposition \ref{baseCase}]
Let $K \subseteq \F_p^n$, $\varepsilon = \frac{1}{4n(n - 1)}$ and $t > 0$ satisfy $p^{t/4} \leq |K| \leq p^{n - \frac{2n + 1}{4(n - 1)}t}$. Assume that $\mathcal{P}(1, \varepsilon, t)$ does not hold for some non-degenerate set $E \subseteq Gr(n, n - 1)$ of cardinality $p^t$. The non-degeneracy of $E$ in this case is equivalent to the existence of $W_1, \ldots, W_n \in E$ with $W_1 \cap \ldots \cap W_n = 0$. For the same reason as in the previous proof we may assume that $W_i = \F_p^{i -1} \times \{0\} \times \F_p^{n - i}$. The failure of $\mathcal{P}(1, K, \varepsilon)$ means that  for any subset $K' \subseteq K$ and any $W \in E$ 
$$ \exists K'' \subseteq K' : |\pi^W(K'')| \leq M \text{ and } |K''| \gtrapprox|K'|,$$ where $M$ is some integer satisfying $M \not \gtrapprox p^{\varepsilon t}|K|^\frac{1}{n}$. Using this we can construct the following sequence of subsets: 
\begin{itemize}
    \item Let $K_1 \subseteq K$ be a subset of size $\gtrapprox |K|$ with $|\pi^{W_1}(K_1)| \leq M$.
    \item Let $K_2 \subseteq K_1$ be a subset of size $\gtrapprox |K|$ with $|\pi^{W_2}(K_2)| \leq M$. \\
    ...
    \item Let $K_n \subseteq K_{n - 1}$ be a subset of size $\gtrapprox |K|$ with $|\pi^{W_n}(K_n)| \leq M$.
\end{itemize}
One of the consequences of this construction is that $|\pi^{W_i}(K_n)| \leq M$ for $i = 1, \ldots, n$. Hence, $K_n \subseteq A_1 \times \ldots \times A_n$ for some sets $A_1, \ldots, A_n$ satisfying $|A_i| = M$ ($i = 1, \ldots, n$). 

Once we placed the set $K_n$ inside a cartesian product the proof proceeds in essentially the same way as the proof of Theorem \ref{lineProjections} with $K_n$ in place of $K$. The only difference is that for each $W$ only the projection of a large subset of $K_n$ is bounded (in cardinality) by $M$ rather than the projection of the whole set $K_n$, so in the lower bound we get (using the same notation):

$$\sum_{x \in \text{span}\{v_3, \ldots, v_n\}} \mathcal{I}(A_1 \times A_2, \mathcal{L}_x) \gtrapprox |E'| |K_n| \gtrapprox |E'||K|.$$
The conclusion is, thus, the same with $\gtrsim$ replaced by $\gtrapprox$, namely:
    $$M \gtrapprox \min\{|K|^\frac{1}{n}|E'|^\frac{1}{4n}, |K|^\frac{1}{n - 1}\} \gtrsim \min\{|K|^\frac{1}{n}|E|^\frac{1}{4n(n - 1)}, |K|^\frac{1}{n - 1}\} \gtrsim p^{t \varepsilon} |K|^\frac{1}{n}$$
\end{proof}

\section{Proof of the reductions}
We first establish two properties of non-degenerate sets that will be useful for the proofs of the reductions (cf. Lemma 33 in \cite{He18}).

\begin{proposition}\label{SumofSubspacesProposition}
    Let $E_1 \subseteq Gr(n, m_1)$ and $E_2 \subseteq Gr(n, m_2)$ be two non-degenerate sets and assume that $m_1 + m_2 < n$. Then the set $$E := \{W_1 + W_2 \mid W_1 \cap W_2 = 0 \text{ and } W_i \in E_i \text{ for } i = 1, 2\}$$ is non-degenerate and $|E| \geq \max\{|E_1|, |E_2|\}^\frac{1}{n}$.
\end{proposition}

\begin{proof}
    First we show that $E$ is non-degenerate. Consider any $V \in Gr(n - m_1 - m_2)$. By the non-degeneracy of $E_1$ we know that there exists $W_1 \in E_1$ such that $W_1 \cap V = 0$. Next, using the non-degeneracy of $E_2$ we obtain that there is $W_2 \in E_2$ such that $W_2 \cap (W_1 + V) = 0$. The latter is equivalent to $V \cap (W_1 + W_2) = 0$. It remains to establish that $W_1 + W_2 \in E$, that is, that $W_1 \cap W_2 = 0$. That, of course, immediately follows from $W_2 \cap (W_1 + V) = 0$.

    Now we need to prove the lower bound on the cardinality of $E$. We may assume without loss of generality that $|E_1| \geq |E_2|$.
    \begin{claim}
        For any $W_1 \in E_1$ we have 
        $$W_1 = \bigcap_{\substack{W_2 \in E_2 \\ W_1 \cap W_2 = 0}}(W_1 + W_2)$$
    \end{claim}
    \begin{proof}
        Consider any $v \not \in W_1$. Since $m_1 + m_2 < n$, the non-degeneracy assumption implies that there exists $W_2 \in E_2$ such that $W_2 \cap (W_1 + \text{span} \{v\}) = 0$. This then implies that $v \not \in (W_1 + W_2)$. Since $v$ was an arbitrary vector in the complement of $W_1$ the claim follows.
    \end{proof}
    The claim above shows that any $W_1 \in E_1$ can be written as the intersection of elements of $E$. Dimension considerations then imply that in fact it is the intersection of $n$ elements $W^1, \ldots, W^n$ of $E$. Hence, we obtain an injective map $E_1 \to E^n$ that sends $W_1$ to $(W^1, \ldots, W^n)$. This proves that $|E|^n \geq |E_1|$ as claimed.
\end{proof}

The next proposition is essentially the dual version of Proposition \ref{SumofSubspacesProposition}.

\begin{proposition}\label{CapofSubspacesProposition}
    Let $E_1 \subseteq Gr(n, m_1)$ and $E_2 \subseteq Gr(n, m_2)$ be two non-degenerate sets and assume that $m_1 + m_2 > n$. Then the set $$E := \{W_1 \cap W_2 \mid  \dim(W_1 \cap W_2) = m_1 + m_2 - n \text{ and } W_i \in E_i \text{ for } i = 1, 2\}$$ is non-degenerate and $|E| \geq \max\{|E_1|, |E_2|\}^\frac{1}{n}$.
\end{proposition}
This proposition can be reduced to the previous one by taking the orthogonal complement. To this end, we state the following elementary property of the non-degeneracy condition.
\begin{proposition}
    Let $E \subseteq Gr(n, m)$ be a non-degenerate set. Then the set $$E^\bot := \{W^\bot \mid W \in E\} \subseteq Gr(n, n - m)$$ is also non-degenerate.
\end{proposition}

\begin{proof}
    Consider arbitrary $W, V \in Gr(n, m)$. Using the properties of the orthogonal complement we obtain the following chain of equivalences:
    $$W^\bot \cap V = 0 \Leftrightarrow W^\bot + V = \F_p^n \Leftrightarrow (W^\bot + V)^\bot = 0 \Leftrightarrow W \cap V^\bot = 0.$$
    This together with the non-degeneracy of $E$ and the fact that $\dim(V^\bot) = n - m$ immediately implies that $E^\bot$ is also non-degenerate.
\end{proof}

\begin{proof}[Proof of Proposition \ref{CapofSubspacesProposition}]
    Applying the previous proposition we see that $E_1^\bot$ and $E_2^\bot$ are non-degenerate subsets of $Gr(n, n - m_1)$ and $Gr(n, n - m_2)$ respectively. Since $(n - m_1) + (n - m_2) < n$, applying Proposition \ref{SumofSubspacesProposition} to $E_1^\bot$ and $E_2^\bot$ we get that the set $E':=\{W_1^\bot + W_2^\bot \mid W_i \in E_i \text{ and } W_1^\bot \cap W_2^\bot = 0\}$ is non-degenerate and has cardinality $\geq \max\{|E_1|, |E_2|\}^\frac{1}{n}$. It remains to observe that $$W_1^\bot \cap W_2^\bot = 0 \Leftrightarrow W_1 + W_2 = \F_p^n \Leftrightarrow \dim(W_1 \cap W_2) = m_2 + m_2 - n$$ and that therefore the set $E = (E')^\bot$ is non-degenerate and has the required cardinality.  
\end{proof}

Now we prove Propositions \ref{additionofDimensionsProposition} and \ref{additionofCodimensionsProposition}. The proofs, for the most part, closely follow \cite{He18}, so we keep the details to the minimum.

\begin{proof}[Proof of proposition \ref{additionofDimensionsProposition}]
    Let $E_1 \subseteq Gr(n, n - m_1)$ and $E_2 \subseteq Gr(n, n - m_2)$ be the non-degenerate sets that provide the counterexamples to $\mathcal{P}(m_1, t, \varepsilon)$ and $\mathcal{P}(m_2, t, \varepsilon)$ respectively. Let  $K' \subseteq K$ be an arbitrary subset of $K$. The failure of $\mathcal{P}(m_1, t, \varepsilon)$ means that for any $W_1 \in E$ we have 
    $$|\pi^{W_1}(K'')| \not \gtrapprox |K|p^{\varepsilon t} \text{ for some } K'' \subseteq K' \text{ with }|K''| \gtrapprox |K'|.$$ 
    Next, since the $\mathcal{P}(m_2, t, \varepsilon)$ fails for $E_2$ we know that for each such $K''$ and any $W_2 \in E_2$ there exists a further subset $K'''$ of size $\gtrapprox|K''|$ such that $$|\pi^{W_2}(K''')| \not \gtrapprox |K|^\frac{m_2}{n}p^{t \varepsilon}.$$ 
    To sum up, for every $W_1 \in E_1$ and every $W_2 \in E_2$ there exists a subset $K''' \subseteq K'$ such that $|K'''| \gtrapprox|K'|$ and $$|\pi^{W_i}(K''')| \not \gtrapprox |K|^\frac{m_i}{n}p^{\varepsilon t} \text{ for } i = 1, 2.$$
    
    Consider the set $$E := \{W_1 \cap W_2 \mid W_i \in E_i \text{ and } \dim(W_1 \cap W_2) = n - m_2 - m_2\}.$$ By Proposition \ref{CapofSubspacesProposition} we know that it is a non-degenerate set of cardinality $\geq p^{\frac{t}{n}}$. Now, for any element $W_1 \cap W_2$ of $E$ we have the bound $$|\pi^{W_1 \cap W_2}(K''')| \leq |\pi^{W_1}(K''')| |\pi^{W_2}(K''')| \not \gtrapprox |K|^\frac{m_1 + m_2}{n}p^{2 \varepsilon t}$$ for some $K''' \subseteq K'$ satisfying $|K'''| \gtrapprox|K|$ by Proposition \ref{intersectionBound}. The desired result follows immediately.
\end{proof}

The following is an analog of Proposition 34 from \cite{He18}:

\begin{proposition}\label{sumBound}
Suppose $W_1 \in Gr(n, n - m_1)$, $W_2 \in Gr(n, n - m_2)$ and $W_1 \cap W_2 = 0$. Assume that for some $M_1, M_2 > 0$ we have
\begin{itemize}
    \item $|\pi^{W_1}(K)| \leq   M_1$
    \item $|\pi^{W_2}(K)| \leq M_2$.
\end{itemize}
Then there exists a subset $K' \subseteq K$ satisfying $|K'| \gtrsim |K| / \log p$ such that $$|\pi^{W_1 + W_2}(K')| \lesssim \frac{M_1 M_2}{|K'|} $$
\end{proposition}

\begin{proof}
First, we may apply a linear transformation and assume that we have $W_1 = \text{span}\{v_1, \ldots, v_{n - m_1}\}$ and $W_2 = \text{span}\{v_{n - m_1 + 1}, \ldots, v_{2n - m_1 - m_2}\}$, where $v_1, \ldots, v_n$ is the standard basis of $\F_p^n$. Next, by dyadic pigeonholing we can find a subset $K' \subseteq K$ of cardinality $\gtrsim |K|/\log p$ such that all non-empty slices $K' \cap (x + W_1 + W_2)$ have roughly the same cardinality $\sim \frac{|K'|}{|\pi^{W_1 + W_2}(K')|}$. Applying Lemma 37 from \cite{He18} we get 
$$|K'|^3 \leq |\pi^{W_1}(K')| |\pi^{W_2}(K')| \sum |K' \cap (x + W_1 + W_2)|^2 \lesssim \frac{M_1 M_2|K'|^2}{|\pi^{W_1 + W_2}(K')|}$$
where the sum is over all elements $x + W_1 + W_2$ of $\F_p^n/(W_1 + W_2)$.
Rearranging, $|\pi^{W_1 + W_2}(K')| \lesssim \frac{M_1 M_2}{|K'|} $.
\end{proof}

The proof of  Proposition \ref{additionofCodimensionsProposition} is now almost the same as the proof of Proposition \ref{additionofDimensionsProposition}.

\begin{proof}[Proof of proposition \ref{additionofCodimensionsProposition}]
    Let $E_1 \subseteq Gr(n, n - m_1)$ and $E_2 \subseteq Gr(n, n - m_2)$ be the non-degenerate sets that provide the counterexamples to $\mathcal{P}(m_1, t, \varepsilon)$ and $\mathcal{P}(m_2, t, \varepsilon)$ respectively. Let  $K' \subseteq K$ be an arbitrary subset of $K$. Repeating the same procedure as in the first part of the proof of Proposition \ref{additionofDimensionsProposition} we arrive at the following conclusion:
    for every $W_1 \in E_1$ and every $W_2 \in E_2$ there exists a subset $K''' \subseteq K'$ such that $|K'''| \gtrapprox|K'|$ and $$|\pi^{W_i}(K''')| \not \gtrapprox |K|^\frac{m_i}{n}p^{\varepsilon t} \text{ for } i = 1, 2.$$
    
    Consider the set $$E := \{W_1 + W_2 \mid W_i \in E_i \text{ and } W_1 \cap W_2 = 0\}.$$ By Proposition \ref{CapofSubspacesProposition} we know that it is a non-degenerate set of cardinality $\geq p^{\frac{t}{n}}$. Now, for any element $W_1 + W_2$ of $E$ we have the bound $$|\pi^{W_1 + W_2}(K'''')| \not \gtrapprox |K|^\frac{m_1 + m_2 - n}{n}p^{2 \varepsilon t}$$ for some $K'''' \subseteq K'$ satisfying $|K''''| \gtrapprox |K|$ by Proposition \ref{sumBound}. The desired result follows immediately.
\end{proof}

\section{Proof of Theorem \ref{BourgainTypeProjections}}
Now we will prove Theorem \ref{BourgainTypeProjections} using Propositions \ref{additionofDimensionsProposition} and \ref{additionofCodimensionsProposition}. The following elementary proposition will enable us to apply them.

\begin{proposition}\label{sequenceProposition}
    Let $S$ be a subset of $\{1, \ldots, n - 1\}$ that it is not contained in $d \mathbb{Z}$ for any $d \neq 1$ that divides $n$. Then there is a sequence of integers $m_1, \ldots, m_k$ such that $1 \leq m_i \leq n - 1$, $m_k = 1$ and for any $1 \leq i \leq m$ we have one of the following:
    \begin{itemize}
        \item $m_i \in S$,
        \item $m_i = m_{j_1} + m_{j_2}$ for some $j_1, j_2 < i$,
        \item $m_i = m_{j_1} + m_{j_2} - n$ for some $j_1, j_2 < i$.
    \end{itemize}
\end{proposition}

\begin{proof}
Let $S':= \{m_k \mid \text{there exists a sequence } m_1, \ldots, m_k \text{ as above}\}$ and let $m \in S'$ be the element that minimizes $\min\{m, n - m\}$.

\textit{Case 1}: $m \leq \frac{n}{2}$. We claim that in this case 
\begin{itemize}
    \item $n$ is divisible by $m$,
    \item any element of $S'$ is divisible by $m$
\end{itemize}
To see the first claim, note that if $m$ does not divide $n$, then we can write $n = qm + r$, where $0 < r < m$. Then $qm$ also belongs to $S'$, and $\min(qm, n - qm) = r < m = \min(m, m - n)$, which yields a contradiction with the choice of $m$. To prove the second claim consider any element $m' \in S$. Then $m' \leq n - m$. If $m'$ is not divisible by $m$, we  can write $n = m' + qm + r$, where $q \geq 1$ and $0 < r < m$. Since $m'+ qm$ belongs to $S'$, we once again get a contradiction.\\
\textit{Case 2}: $m \geq \frac{n}{2}$. We claim that in this case 
\begin{itemize}
    \item $n$ is divisible by $n - m$,
    \item any element of $S'$ is divisible by $n - m$
\end{itemize}
In this situation, the proof is the same as in the first case except that we use the fact that $S'$ is closed under $(m_1, m_2) \mapsto m_1 + m_2 - n = n - ((n - m_1) + (n - m_2))$ whenever $m_1 + m_2 > n$ (in other words, we apply the same argument to $n - S'$). \\
We can conclude the proof of the proposition by observing that $m$ has to be $1$ or $n - 1$ since $S$ is not contained in $d \mathbb{Z}$ for any non-trivial divisor of $n$. If $m = 1$, the desired result follows by the definition of $S'$, and if $m = n - 1$, then we can obtain $1$ by a repeated application of the operation  $(m_1, m_2) \mapsto m_1 + m_2 - n$.
\end{proof}

Now we are ready to prove theorem \ref{BourgainTypeProjections}.

\begin{proof}[Proof of Theorem \ref{BourgainTypeProjections}]
    Let $\varepsilon_0 = \frac{1}{4n(n - 1)(2n)^{n - 2}}$. Fix $K \subseteq \F_p^n$ and $t > 0$ such that $|K|p^{\frac{(2n + 1)t}{4(n - 1)n^{n - 2}}} \leq p^n$ and $ p^{\frac{t}{4 n^{n - 2}}} \leq |K|$. Let $S$ be the set of $m$ such that $\mathcal{P}(m, t, \varepsilon_0)$ does not hold. If $S$ is not contained in $d \mathbb{Z}$ for a divisor $1 < d < n$ of $n$, then Proposition \ref{sequenceProposition} yields a sequence of numbers $m_1, \ldots, m_k$ satisfying the conditions above. Of course, we may assume that the numbers $m_1, \ldots, m_k$ do not repeat and, in particular, $k \leq n - 1$.
    Now repeated application of Propositions \ref{additionofDimensionsProposition} and \ref{additionofCodimensionsProposition} will yield a contradiction with the $m = 1$ case via the following claim.
    \begin{claim}
        $\mathcal{P}(m_i, t/n^{i - 1}, (2n)^{i - 1}\varepsilon_0)$ fails for each $i = 1, \ldots, k$.
    \end{claim}
    
    \begin{proof}
        We prove this by induction on $i$. There are several cases.
        
        \textit{Case 1}: $m_i \in S$. 

        Then $\mathcal{P}(m_i, t, \varepsilon_0)$ fails and it follows from the definition that $\mathcal{P}(m_i, t/n^{i - 1}, n^{i - 1}\varepsilon_0)$ also fails. This implies the claim.

        \textit{Case 2}: There are $j_1, j_2 < i$ with $m_{j_1} + m_{j_2} = m_i$.

        We may assume that $j_1 < j_2$. By the induction hypothesis we have that $\mathcal{P}(m_{j_1}, t/n^{j_1 - 1}, (2n)^{j_1 - 1}\varepsilon_0)$ and $\mathcal{P}(m_{j_2}, t/n^{j_2 - 1}, (2n)^{j_2 - 1}\varepsilon_0)$ fail. Again, from the definition it follows that $\mathcal{P}(m_{j_1}, t/n^{j_2 - 1}, (2n)^{j_2 - 1}\varepsilon_0)$ also fails. Now we can apply Proposition \ref{additionofDimensionsProposition} to get that $\mathcal{P}(m_{j_1} + m_{j_2}, t/n^{j_2}, (2n)^{j_2}\varepsilon_0)$ which implies the claim for $i$ because $j_2 \leq i - 1$.

        \textit{Case 3}: There are $j_1, j_2 < i$ with $m_{j_1} + m_{j_2} - n = m_i$.

        In this case, the proof is the same as in Case 2 except that we use Proposition \ref{additionofCodimensionsProposition} instead of Proposition \ref{additionofDimensionsProposition}.
    \end{proof}
    
    Consequently, since $k \leq n - 2$ and $m_k = 1$ we have that $\mathcal{P}(1, t/{n^{n - 2}}, (2n)^{n - 2}\varepsilon_0)$ does not hold. This yields a contradiction with Proposition \ref{baseCase} because $(2n)^{n - 2}\varepsilon_0 = \frac{1}{4n(n - 1)}$, $|K| \leq  p^{n -   \frac{2n + 1}{4(n - 1)} \cdot (t/n^{n - 2})}$ and $p^{\frac{t}{4n^{n - 2}}} \leq |K|$. 

    Thus we proved that there exists a divisor $d > 1$ of $n$ such that $\mathcal{P}(m, t, \varepsilon_0)$ holds when $m$ is not divisible by $d$. It remains to observe that $\mathcal{P}(m, t, \varepsilon_0)$ implies that
     $$\exists W \in E \text{ such that } |\pi^{W}(K)| \gtrsim |K|^{\frac{m}{n}}|E|^{\varepsilon} $$ for any $E$ satisfying the condition of Theorem \ref{BourgainTypeProjections} and any $\varepsilon < \varepsilon_0$.
\end{proof}

\section{Projections of very small sets and proof of Theorem \ref{improvementTheorem}}
We begin this section by proving theorem \ref{smallSetProjections}. It follows from the incidence bound for very small set established in \cite{Gro13} in essentially the same way as theorem \ref{planarProjections} follows from the incidence bound of Stevens. Technically, the bound in \cite{Gro13} is given for sets where the number of lines and the number of points is the same, so for completeness we include the proof of the general case here (which is no different from the proof in \cite{Gro13}).

\begin{proposition}
    Let $P, \mathcal{L}$ be sets of points and lines respectively in $\F_p^2$ satisfying $|P|, |\mathcal{L}| \leq \frac{1}{5}(\log_2 \log_6 \log_{18} p - 1)$. Then $\mathcal{I}(P, \mathcal{L}) \lesssim |P|^{\frac{2}{3}}|\mathcal{L}|^\frac{2}{3} + |P| + |\mathcal{L}|$.   
\end{proposition}
\begin{proof}
    For each line in $\mathcal{L}$ fix an equation $ax + by + c = 0$ defining this line. Take $A := \{ x, y \mid (x, y) \in P\} \cup \{a, b, c \mid \{ax + by + c = 0\} \in \mathcal{L}\}$ to be the set of all coordinates of points in $P$ and all coefficients of the equations defining the lines in $\mathcal{L}$. By Theorem 3 from \cite{Gro13} there exists an injective map $\phi: A \to \mathbb{C}$ such that $$\phi(z_1) \phi(z_2) + \phi(z_3)\phi(z_4) + \phi(z_5) = 0 \text{ if and only if } z_1 z_2 + z_3 z_4 + z_5 = 0$$ whenever $z_1, z_2, z_3, z_3, z_5 \in A$. Thus, any point $p = (x, y) \in P$ defines a points $\phi(p) := (\phi(x), \phi(y)) \in \mathbb{C}$ and any line $\ell = \{ax + by + c = 0\} \in \mathcal{L}$ defines a line $\phi(\ell) := \{\phi(a)x + \phi(b)y = \phi(c)\}$ in $\mathbb{C}$. In particular, for $p = (x, y) \in P$ and $\ell = \{ax + by  + c = 0\} \in \mathcal{L}$  by taking $z_1 = a, z_2 = x, z_3 = b, z_4 = y, z_5 = c$ we see that $p \in \ell \Leftrightarrow \phi(p) \in \phi(\ell)$. Applying the Szemeredi-Trotter type incidence bound over $\mathbb{C}$ established by Toth in \cite{Tot03} we obtain that $\mathcal{I}(P, \mathcal{L}) \lesssim |P|^{\frac{2}{3}}|\mathcal{L}|^\frac{2}{3} + |P| + |\mathcal{L}|$.
\end{proof}

Now we can prove Theorem \ref{smallSetProjections}

\begin{proof}[Proof of Theorem \ref{smallSetProjections}]
    Assume that for each $W \in E$ we have $|\pi^W(K)| \leq M$ for some number $M > 0$. This means that for each $W \in E$ there is a set of at most $M$ translates of $W$ that cover $K$. Let $\mathcal{L}_W$ be these translates and let $\mathcal{L} = \bigcup_{W \in E} \mathcal{L}_W$. 
    
    Applying the previous theorem with $P = K$ and $\mathcal{L}$ we get $$\mathcal{I}(K, \mathcal{L}) \lesssim |K|^\frac{2}{3}|\mathcal{L}|^\frac{2}{3} + |K| + |\mathcal{L}| \leq  |K|^\frac{2}{3}|E|^\frac{2}{3}M^\frac{2}{3} + |K| + |E|M.$$ On the other hand $\mathcal{I}(K, \mathcal{L}) \geq |K||E|$, so either $M \gtrsim |K|$ or $|K|^\frac{2}{3}|E|^\frac{2}{3}M^\frac{2}{3} \gtrsim |K||E|$ implying $M \gtrsim |K|^\frac{1}{2} |E|^\frac{1}{2}$.
\end{proof}

Next, we prove Theorem \ref{improvementTheorem}. We will use the following bound on the number of subspaces non-trivially intersecting a given subspace.

\begin{proposition} \label{subspaceCounting}
    For any $m, m'$ with $m + m' \leq n$ and any $V \in Gr(n, m')$ we have $$ |\{W \in Gr(n, m) \mid W \cap V \neq 0\}| \leq (1 + o(1))p^{m' - 1 + (m - 1)(n - m)}.$$
\end{proposition}
\begin{proof}
    Any subspace $W$ that intersects $V$ non-trivially has a basis $(w_1, \ldots, w_m)$ with $w_1 \in V$. The number of  $m$-tuples $(w_1, \ldots, w_m) \in V \times (\F_p^n)^{m - 1}$ of linearly independent vectors can be computed as $$(p^{m'} - 1)(p^n - p) \cdot \ldots \cdot(p^n - p^{m - 1}) = (1 + o(1)) p^{m' + (m - 1)n}.$$
    On the other hand, the number of bases $(w_1, \ldots, w_m)$ of a given subspace $W$ with $w_1 \in V$ is at least 
    $$(p  - 1)(p^{m} - p) \cdot \ldots \cdot (p^{m} - p^{m - 1}) = (1 + o(1)) p^{1 + (m - 1)m}$$ 
    assuming $W$ intersects $V$ non-trivially. Therefore, the total number of such subspaces is bounded by
    $$(1 + o(1))\frac{p^{m' + (m - 1)n}}{p^{1 + (m - 1)m}} = (1 + o(1))p^{m' - 1 + (m -1)(n - m)}.$$
\end{proof}
Theorem \ref{improvementTheorem} will be deduced from the following proposition. 
\begin{proposition}\label{improvementProposition}
    Let $n \geq 2$. Fix a parameter $\delta > 0$ and integers $1 \leq m, d \leq n - 1$. Let $E \subseteq Gr(n, 0) \cup Gr(n, 1) \cup \ldots \cup Gr(n, m)$ be a set of subspaces, $K \subseteq \F_p^n$ a set of points and $k \geq 2$ an integer. Assume that the following is satisfied:
    \begin{itemize}
        \item $p^\delta \leq |K| \leq p^{d + 1 - \delta}$,
        \item for any set $S \subseteq Gr(n, 1)$ of lines in $\F_p^n$ with $|S| \leq k$ there exists $\ell \in S$ and $W \in E$ such that $W \cap \ell = 0$,
        \item For any subspace $V \subseteq \F_p^n$ with $\dim V \leq d$ there exists $W \in E$ such that $W \cap V = 0$.
    \end{itemize}
    Then for sufficiently small $\varepsilon = \varepsilon(n, \delta)> 0$ depending only on $n$ and $\delta$ we have 
    $$\exists W \in E: |\pi^W(K)| \gtrsim |K|^\frac{d}{1 + dm}k^\varepsilon.$$
\end{proposition}
\begin{remark}
    The above assumptions on $E$ are satisfied if, for example, $0 \in E$.
\end{remark}
\begin{proof}
    We will prove this by induction on $n$. When $n = 2$ the result follows from Theorem \ref{lineProjections}. Assume that $n \geq 3$ and let $\varepsilon  = \varepsilon(n, \delta)> 0$ be sufficiently small. We consider two cases.

    \textit{Case 1}: There is $W_0 \in E$ with $\dim W_0 \geq 2$. 

    In this case, consider the projection $\pi^{W_0}(K)$ of $K$. Either it has the desired size, or there is a fiber $K \cap (x + W_0)$ of cardinality at least $|K|^\frac{1 + d(m - 1)}{1 + dm}k^{-\varepsilon}$. If the latter occurs, we can apply the induction hypothesis to this fiber. To be more specific, define $\tilde{K} := K \cap (x + W_0)$ and set $\tilde{E}:= \{W \cap W_0 \mid W \in E, W_0 \not \subseteq W\}$. The set $\tilde{E}$ has the same properties as $E$, namely that 
    \begin{itemize}
        \item for any set $S \subseteq Gr(\dim W_0, 1)$ of at most $k$ lines in $W_0$ there is $\tilde{W} \in \tilde{E}$ and $\ell \in S$ such that $\tilde{W} \cap \ell = 0$
        \item for any subspace $V \subseteq W_0$ with $\dim V \leq d$ there exists $\tilde{W} \in \tilde{E}$ such that $W \cap V = 0$.
    \end{itemize}
    As for $\tilde{K}$, we still have the lower bound $$|\tilde{K}| \geq p^{\frac{1 + d(m - 1)}{1 + dm}\delta} k^{-\varepsilon} \gtrsim p^{\delta'} \text{ for } \delta'= \delta/10$$ assuming that $\varepsilon$ is sufficiently small. The upper bound also holds: 
    $$|\tilde{K}| \leq |K| \leq p^{d + 1 - \delta} \leq p^{d + 1 - \delta'}$$
    Applying the induction hypothesis to $\tilde{K}$ (viewed as a subset of $\F_p^{\dim W_0}$) and $\tilde{E} \subseteq Gr(\dim W_0,0) \cup \ldots \cup Gr(\dim W_0, \dim W_0 - 1)$ we see that for some $\varepsilon_0 = \varepsilon_0(\dim W_0, \delta')$ we have the bound 
    $$|\pi^{\tilde{W}}(\tilde{K})| \gtrsim |\tilde{K}|^\frac{d}{1 + d(m - 1)}k^{\varepsilon_0} \geq |K|^\frac{d}{1 + dm} k^{\varepsilon_0 - \frac{\varepsilon}{1 + d(m - 1)}}$$
    The last quantity exceeds $|K|^\frac{d}{1 + dm}k^\varepsilon$ if $\varepsilon$ is sufficiently small. 

    \textit{Case 2}: $E \subseteq Gr(n, 0) \cup Gr(n, 1)$. 

    If $0 \in E$, then $|\pi^{0}(K)| = |K|$ and the result is trivially true, so we will focus on the case $E \subseteq Gr(n, 1)$. Let $V := span E$. The assumptions on $E$ imply that $\dim V \geq d + 1$. Consider the slices $K_x := K \cap (x + V)$. We may assume that all non-empty slices $K_x$ satisfy $|K_x| \geq p^{\delta'}$ for $\delta' = \frac{\delta}{10(d + 1)}$. Indeed, this is because if $|\pi^{W}(K)| \leq |K|^\frac{d}{d + 1}k^\varepsilon$ for some $W \in E$, then the total cardinality of small slices (with $|K_x| \leq p^{\delta'}$) is at most $|K|^{\frac{d}{d + 1}}k^\varepsilon p^{\delta'} \leq |K|/2$ provided $\varepsilon$ is chosen to be sufficiently small. Thus, we may discard all small slices and assume that $|K_x| \geq p^{\delta'}$ for all non-empty $K_x$.
    
    Now we we will apply Theorem \ref{BourgainTypeProjections} to each slice $K_x \subseteq x + V$. The non-degeneracy condition in the case $E \subseteq Gr(\dim V, 1)$ is equivalent to $E$ spanning the whole space $V$, which indeed holds. Therefore, viewing each $K_x$ as a subset of $V$ we see that for some small $\varepsilon = \varepsilon(\dim V, \delta') > 0$ we have the bound $$|\pi^W(K_x)| \gtrsim |K_x|^\frac{\dim V - 1}{\dim V}|E|^{\varepsilon} \geq |K_x|^\frac{d}{d + 1}k^{\varepsilon}$$  for more than half of the elements $W$ of $E$. This is because otherwise we could apply Theorem \ref{BourgainTypeProjections} to the set of $W$ for which the bound does not hold obtaining a contradiction. 
    
    Since for each $K_x$ the estimate above is true for more than half of the elements of $E$, there exists $W \in E$ such that $$ |\pi^W(K_x)| \gtrsim |K_x|^\frac{d}{d + 1}k^{\varepsilon} \text{ whenever } K_x \in \mathcal{S}$$ for a certain collection $\mathcal{S}$ of slices of total mass $\sum_{K_x \in \mathcal{S}} |K_x| \geq |K|/2$. Hence, for this $W$ we have 
    $$|\pi^W(K)| \geq \sum_{\mathcal{S}} |\pi^{W}(K_x)| \gtrsim \sum_{\mathcal{S}} |K_x|^\frac{d}{d + 1}k^{\varepsilon} \geq (\sum_{\mathcal{S}} |K_x|)^\frac{d}{d + 1}k^{\varepsilon} \gtrsim |K|^\frac{d}{d + 1}k^{\varepsilon}.$$
\end{proof}

Theorem \ref{improvementTheorem} is an almost immediate corollary.

\begin{proof}[Proof of Theorem \ref{improvementTheorem}]
    Let $k = \frac{1}{2}|E|/p^{m(n - m) - m}$. Proposition \ref{subspaceCounting} implies that for any set of at most $k$ lines the total number of elements of $Gr(n, n - m)$ containing one of these lines is at most $$(1 + o(1))k p^{m(n - m) - m} < |E|.$$
    The same proposition also shows that the total number of elements of $Gr(n, n - m)$ intersecting a given subspace of dimension at most $d$ is bounded by $$(1 + o(1))p^{m(n - m) - m + d - 1} < |E|.$$
    Hence, we can apply the previous proposition to obtain a lower bound of $$\exists W \in E: |\pi^{W}(K)| \gtrsim |K|^\frac{d}{1 + d(n - m)} k^\varepsilon \gtrsim |K|^\frac{d}{1 + d(n - m)}(|E|/p^{m(n - m) - m})^\varepsilon.$$
\end{proof}

\end{document}